\documentclass{article}
\usepackage[utf8]{inputenc}
\usepackage{amsmath,amssymb,amsthm}
\usepackage{amssymb,dsfont}
\usepackage{cleveref}

\usepackage{amsopn}

\usepackage{epsfig}
\usepackage{pst-all} 
\usepackage{pst-plot} 
\usepackage[space]{grffile} 
\usepackage{algorithm}
\usepackage[noend]{algpseudocode}

\newcommand{\be}{\begin{equation}}
\newcommand{\ee}{\end{equation}}

\newtheorem{definition}{Definition}[section]
\newtheorem{theorem}[definition]{Theorem}

\newtheorem{corollary}[definition]{Corollary}
\newtheorem{remark}[definition]{Remark}

\title{A Multilevel Monte Carlo Estimator for Matrix Multiplication\thanks{NP and YW are grateful to EPSRC for funding this work through the project EP/R041431/1: `Randomness: a resource for real-time analytics'; YW is also funded by The Alan Turing Institute under the EPSRC grant EP/N510129/1 and by EPSRC through the project EP/S026347/1:' Unparameterised multi-modal data, high order signatures, and the mathematics of data science'.}}

\author{Yue Wu \thanks{Mathematical Institute, University of Oxford, Oxford, UK; The Alan Turing Institute, London, UK;} 
 \and Nick Polydorides
 \thanks{School of Engineering, University of Edinburgh, Edinburgh, UK; The Alan Turing Institute, London, UK.}
}

\begin{document}

\maketitle

\begin{abstract}
Inspired by recent developments in multilevel Monte Carlo (MLMC) methods and randomised sketching for linear algebra problems we propose a MLMC estimator for real-time processing of matrix structured random data. Our algorithm is particularly effective in handling high-dimensional inner products and matrix multiplication, and finds applications in computer vision and large-scale supervised learning. 
\end{abstract}



\section{Introduction}
Randomised algorithms for matrix operations are in  general `pass-efficient', and are primarily aimed at problems involving massive data sets that are otherwise cumbersome to process with deterministic algorithms. Pass-efficient implies that the algorithm necessitates only a very small number of passes through the complete data set, but for the cases we consider here such a pass maybe turn out to be impractical due to memory or time restrictions.  In matrix multiplication for example, the B{\scriptsize ASIC}M{\scriptsize ATRIX}M{\scriptsize ULTIPLICATION} algorithm \cite{Drineas06} is considered to be the gold standard. Based on a probability assigned to the columns of a matrix $A$, and respectively the rows of a matrix $B$, it approximates the product $AB$ through re-scaling the outer products of some sampled columns of $A$ with the corresponding rows of $B$ via a sampling-and-rescaling matrix operator. Variants of the B{\scriptsize ASIC}M{\scriptsize ATRIX}M{\scriptsize ULTIPLICATION} algorithm were published in  \cite{Eriksson11}, \cite{HI15}, \cite{Wu18}, exploiting different types of information available on the elements of the matrices involved. In particular, the algorithm in \cite{Eriksson11} addresses the case where the probability distributions of the elements are known a priori to devise an importance sampling strategy based on B{\footnotesize ASIC}M{\footnotesize ATRIX}M{\footnotesize ULTIPLICATION} that minimizes the expected value of the variance. The algorithm was shown to be effective when implemented with the optimized sampling probabilities, particularly so in comparison to the estimators resulting from uniform sampling. This result indeed extends B{\footnotesize ASIC}M{\footnotesize ATRIX}M{\footnotesize ULTIPLICATION} to a random variable setting and can be applied to many query matching with information retrieval applications \cite{Eriksson11}. However, designing the optimized probabilities relies exclusively on the knowledge of the probability distributions of the matrix elements, which limits its applicability to the cases where such information is a priori available. Conversely, it can be argued that B{\footnotesize ASIC}M{\footnotesize ATRIX}M{\footnotesize ULTIPLICATION} with uniform probabilities becomes more appealing when dealing with real-time random matrix multiplication tasks, where distributions change dynamically. In batch processing for instance, the task at hand is to evaluate the expectation of the multiplication or indeed a functional of a matrix product at any given time, a formidable task in terms of the required speed and accuracy. To accelerate the time-dependent training of large-scale kernel machines for example, the evaluation of a kernel function is identified as and approximated through the expectation of a random inner product via some randomised feature map \cite{Kar12}, \cite{Rahimi08}. In this case, coupling a standard Monte Carlo method (MC) and B{\footnotesize ASIC}M{\footnotesize ATRIX}M{\footnotesize ULTIPLICATION} with uniform probabilities may satisfy the speed specifications but compromise the accuracy of the result. A more prudent alternative is to employ a multilevel Monte Carlo method, similar to the one developed in \cite{Giles08} instead of MC. 

MLMC was initially conceived for reducing the cost of computing the expected value of a financial derivative whose payoff depends upon the solution of a stochastic differential equation (SDE). The framework in \cite{Giles08} generalizes Kebaier's approach in \cite{Kebaier05} to multiple levels, using a geometric sequence of different time step sizes. In doing so it reduces substantially the computational cost of MC by taking most of the samples on coarse grids at low cost and accuracy, and only a few samples on finer computationally expensive grids that lead to solutions of high accuracy. Over time, MLMC has grown in scope and found a wide range of applications in the broad area of SDEs, SPDEs, for stochastic reaction networks and inverse problems  \cite{Teckentrup}, while further variants have been developed in the form of multilevel quasi-Monte Carlo estimators \cite{Giles09} and multilevel sequential Monte Carlo samplers \cite{Beskos}. For an overview on MLMC we refer the reader to the excellent survey \cite{Giles15}. Therein the author emphasizes that the multilevel theorem allows to use other estimators as long as they satisfy some specific conditions. This theorem lays the foundation for the algorithm proposed in this paper. A closely related work \cite{Bierig} considers the MLMC estimate for approximating the mean field of a nonlinear PDE, providing a theoretical framework in separable Hilbert spaces. Although there is clearly no actual time stepsize in the matrix multiplication context, we can draw an analogy between the term \emph{time stepsize} in numerical analysis for differential equations and the term \emph{the size of the sampled index set} in randomised linear algebra. As anticipated in a convergent MLMC scheme, the numerical estimation error shrinks with decreasing time stepsize. Similarly, due to the law of large numbers, increasing the size of index samples will decrease the expected squared Frobenius approximation error as shown in Lemma 4 of the seminal work \cite{Drineas06}. Therefore we claim that a random strategy for matrix multiplication with fewer index samples is analogous to using a ``coarser grid" in the PDE setting. This observation is crucial to our construction of MLMC estimators for matrix multiplication.

In Section \ref{sec:inner} below we begin by discussing the simpler case of calculating `on the fly' the expectation of the inner product between large random vectors. We first consider the B{\footnotesize ASIC}M{\footnotesize ATRIX}M{\footnotesize ULTIPLICATION} algorithm with uniform probability and proceed to review the main results for the inner product from \cite{Eriksson11}. We then introduce the important quantities \emph{base number} $M\in \mathbb{N}$ and \emph{level size} $L\in\mathbb{N}$ based on which the MLMC estimator (c.f. \eqref{eqn:hatYl} and \eqref{eqn:hatY}) is constructed via inner product approximations with index sample sizes $M^0,M^1,\ldots,M^L$. In this context, the approximation on the `finest grid' corresponds to the inner product realization with $M^L$ samples. Here we note the distinction between samples and indices, in that since we are sampling with replacement, taking $M^L>n$ samples does not imply sampling all $n$ indices. Given that the variance of the approximated inner product is proportional to $M^{-l}$ for $l\in \{1,\ldots,L\}$ (c.f. Theorems \ref{thm:minvar} and Theorem \ref{thm:eqvar}), the complexity of the proposed MLMC estimator for a functional of the inner product conditioned on certain features of the underlying approximation can be treated similarly as the case $\beta=1$ of Theorem 3.1 in \cite{Giles08}. This result is revisited in Theorem \ref{thm:eqvar} where a comparison with standard MC is attempted. Corollary \ref{cor:comp} discusses the computational complexity of our MLMC estimator using Theorem \ref{thm:eqvar}. At the end of Section \ref{sec:inner}, we comment on the optimal choice of base number $M$ following the reasoning in \cite{Giles08}. 

In Section \ref{sec:matrix} we extend our approach to matrix multiplication, adapting Theorems \ref{thm:minvar} and \ref{thm:eqvar} accordingly. It is worth mentioning that, because the approximation error (c.f. Theorems \ref{thm:minvarhigh} and \ref{thm:eqvarhigh}) is measured in expectation as a Frobenius norm, for the analysis the matrices are considered transformed in vector form prompting a new definition of `variance' for the vectorized matrices denoted as $\mathbb{V}_{\Vert}$. Further, Theorem \ref{thm:comphigh} discusses the complexity and Corollary \ref{cor:comphigh} validates the complexity of the MLMC estimator for matrix multiplication. The implementation of our method is presented as Algorithm \ref{alg:MLMCmatrix}. Finally, in Section \ref{sec:num} we present two simple numerical experiments to illustrate the performance of the MLMC estimator in comparison with the standard MC one. By making appropriate choices for $M$ and $L$ parameters, the proposed MLMC estimator outperforms the MC estimator in terms of accuracy as well as speed and computational efficiency.

\section{Inner product} \label{sec:inner}

We define $\mathbf{T}$ as a countable collection of discrete time points and set $t\in \mathbf{T}$. Let $\mathbf{a}(t)$ and $\mathbf{b}(t)$ be two random vectors of length $n$, whose elements are drawn from some unknown, perhaps different, probability distributions, say $\mathbf{a}(t) \sim \mathcal{L}_{\mathbf{a}(t)}$ and $\mathbf{b}(t) \sim \mathcal{L}_{\mathbf{b}(t)}$. Here and throughout this paper, $n$ is assumed to be extremely large such that evaluating the inner product of $\mathbf{a}(t)^T\mathbf{b}(t)$ is deemed impractical if at all possible. Consider that there is a need to compute $\mathbb{E}_{\mathbf{a}(t),\mathbf{b}(t)}[f(\mathbf{a}(t)^T\mathbf{b}(t))]$ on demand, at different times, where $f$ is a Lipchitz function with Lipchitz constant $C_f$ and $\mathbb{E}_{\mathbf{a}(t),\mathbf{b}(t)}$ is the expectation under $\mathcal{L}_{\mathbf{a}(t)}$ and $\mathcal{L}_{\mathbf{b}(t)}$. For the sake of notational simplicity, the argument $(t)$ is suppressed in the notation but assumed implicitly in all of the quantities introduced above. 

Indeed the task at hand consists of two main parts: approximating $\mathbf{a}^T\mathbf{b}$ in an efficient and accurate manner and approximating its expected value in the spirit of Monte Carlo methods. To tackle the first issue, the random sampling method for inner product estimation presents a viable option. Suppose there is a sampling distribution $\xi:=\{\xi_j\}_{j=1}^{n}$ with $\sum_{j=1}^n \xi_j=1$ such that each index $j\in [n]$, where $[n]:=\{1,2,\ldots,n\}$, can be drawn with an assigned positive probability $\xi_j$. Further suppose we {\color{black}\emph{fix}} a `base' number $M\in \mathbb{N}$ and collect $M^L$, $L\in \mathbb{N}$, independent and identically distributed index samples as {\color{black}an index sequence $(r_1,\ldots, r_{M^L})$} according to $\xi$. We shall refer to these collected $M^L$ indices,{\color{black} or equivalently, the sequence $(r_1,\ldots, r_{M^L})$,} as a sample \emph{realisation}. Then denote by $S_L$ the \emph{sampling-and-rescaling matrix} of size $n\times M^L$ such that elements of $\mathbf{a}$ and $\mathbf{b}$ at the $M^L$ index samples will be used for approximating the inner product of $\mathbf{a}^T\mathbf{b} ${\color{black}. That} is,
\begin{align}\label{eqn:innersketch}
\widehat{{\mathbf{a}^T\mathbf{b}}}=\mathbf{a}^TS_L S^T_L\mathbf{b} =\frac{1}{M^L}\sum_{i=1}^{M^L} \frac{1}{\xi_{r_i}}\mathbf{a}_{r_i}\mathbf{b}_{r_i}{\color{black}:=X_{L}(\xi),} 
\end{align}
{\color{black}where $X_{L}(\xi)$ denotes a scalar random variable that approximates the target $\mathbf{a}^T\mathbf{b}$ using $M^L$ samples from $\xi$, emphasizing its dependence on $\xi$.}
{\color{black} Previous research have shown that $X_{L}(\xi)$ is an unbiased estimator of ${\mathbf{a}^T\mathbf{b}}$ under the sampling distribution $\xi$}. The performance of the approximation {\color{black} when the vector elements $\mathbf{a}$ and $\mathbf{b}$ are known only up to their distributions} can be assessed through quantifying the variance of the estimator. The minimum variance is attained when sampling according to the distribution given by the following theorem from \cite{Eriksson11}.  
\begin{theorem}\label{thm:minvar}
If the vector elements $\mathbf{a}_j$ and $\mathbf{b}_j$ are independent random variables, $j\in [n]$, with finite and nonzero moments $\mathbb{E}_{\mathbf{a},\mathbf{b}}[\mathbf{a}_j^2\mathbf{b}_j^2]$, then the probability $\xi^{*}$ with elements
\begin{align}\label{eqn:optpro}
    \xi^{*}_j=\frac{\sqrt{\mathbb{E}_{\mathbf{a},\mathbf{b}}[\mathbf{a}_j^2\mathbf{b}_j^2]}}{\sum_{i=1}^n\sqrt{\mathbb{E}_{\mathbf{a},\mathbf{b}}[\mathbf{a}_i^2\mathbf{b}_i^2]}},
\end{align}
minimizes the expected value of the variance in \eqref{eqn:innersketch}, that is,
\begin{align}\label{eqn:minvar}
   {\color{black}   \min_{\xi}   \mathbb{E}_{\mathbf{a},\mathbf{b}}[\textbf{Var}[X_{L}(\xi)]]= \mathbb{E}_{\mathbf{a},\mathbf{b}}[\textbf{Var}[X_{L}(\xi^{*})]]}:=\frac{\mu}{M^L},
\end{align}
where {\color{black} $\textbf{Var}$ is the variance under $\xi$ and} $\mu=\mathbb{E}_{\mathbf{a},\mathbf{b}}\Big[\sum_{i=1}^n\frac{\mathbf{a}_i^2\mathbf{b}_i^2}{\xi^{*}_i}-(\mathbf{a}^T\mathbf{b})^2\Big]$.
\end{theorem}
Sampling with $\xi^{*}$ is clearly not practical when we have no knowledge about the distributions of $\mathbf{a}$ and $\mathbf{b}$ in advance, hence a plausible convenient alternative is to use a uniform probability over the index set
\begin{align}\label{eqn:unique}
\xi^u_j=\frac{1}{n},\ \ \  j\in [n],   
\end{align}
with variance as follows.
\begin{theorem}\label{thm:eqvar}\cite{Eriksson11}
Assume the same setting as in Theorem \ref{thm:minvar} but with probability $\xi^u$ defined in \eqref{eqn:unique}, then the variance is 
\begin{align}\label{eqn:eqvar}
    {\color{black}  \mathbb{E}_{\mathbf{a},\mathbf{b}}[\textbf{Var}[X_{L}(\xi^u)]]=\mathbb{E}_{\mathbf{a},\mathbf{b}}[\textbf{Var}[X_{L}(\xi^{*})]]}+\frac{n\nu}{M^L}=\frac{n\nu+\mu}{M^L},
\end{align}
where 
$$\nu=\sum_{i=1}^n\Big(\sqrt{\mathbb{E}_{\mathbf{a},\mathbf{b}}[\mathbf{a}_i^2\mathbf{b}_i^2]}-\frac{1}{n}\sum_{j=1}^n\sqrt{\mathbb{E}_{\mathbf{a},\mathbf{b}}[\mathbf{a}^2_j\mathbf{b}^2_j]}\Big)^2.$$
\end{theorem}

Typically one may consider approximating the expectation using a standard MC method that simulates $\mathbb{E}_{\mathbf{a},\mathbf{b}}[f(\mathbf{a}^T\mathbf{b})]$. In this instance, the quantity of interest, say $P$, can then be estimated by \eqref{eqn:innersketch} with a uniform  probability \eqref{eqn:unique} and MC as
\begin{align}\label{eqn:P}
\begin{split}
\mathbb{E}_{\mathbf{a},\mathbf{b}}[P]:=\mathbb{E}_{\mathbf{a},\mathbf{b}}[f(\mathbf{a}^T\mathbf{b})]  \approx\frac{1}{N}\sum_{k=1}^N f\big((\mathbf{a}^{(k)})^T\mathbf{b}^{(k)} \big){\color{black}=\frac{1}{N}\sum_{k=1}^N f\big(X^{(k)}_{L}(\xi^u) \big)}:=\hat{P},
\end{split}
\end{align}
where $N$ is the number of realisations for $M^L$ many index samples {\color{black} or equivalently, an index sequence of length $M^L$}.  In this case the mean square error (MSE) for the estimate $\hat P$ turns out to be
\begin{align}\label{eqn:mse1}
\begin{split}
\mathbb{E}\big[\big(\hat P-\mathbb{E}[P]\big)^2\big]& =\mathbb{E}\big[\big(\hat P-\mathbb{E}[\hat P]\big)^2\big]+\big(\mathbb{E}[P]-\mathbb{E}[\hat P]\big)^2\\
&{\color{black}=\mathbb{E}\big[\big(\hat P-\mathbb{E}[\hat P]\big)^2\big]+\big(\mathbb{E}[f(\mathbf{a}^T\mathbf{b})-f(X_{L}(\xi^u))]\big)^2},
\end{split}
\end{align}
where $\mathbb{E}$, and also $\mathbb{V}$ that appears in the sequel, (without subscripts) denote respectively the expectation and the variance under $\mathcal{L}_{\mathbf{a}}$, $\mathcal{L}_{\mathbf{b}}$ and $\xi^{u}$. The last term in \eqref{eqn:mse1}, for a fixed $L$, characterizes the bias and can be bounded by
\begin{align*}
&{\color{black}\big(\mathbb{E}[f(\mathbf{a}^T\mathbf{b})-f(X_{L}(\xi^{u}))]\big)^2  \leq \mathbb{E}\big[\big(f(\mathbf{a}^T\mathbf{b})-f(X_{L}(\xi^{u}))\big)^2\big]}\\
&{\color{black}=\mathbb{E}\big[\big(f(\mathbf{a}^T\mathbf{b})-f(\mathbf{a}^TS_L S_L^T \mathbf{b})\big)^2\big]\leq C_f^2 \mathbb{E}[|\mathbf{a}^T(I-S_L S_L^T) \mathbf{b}|^2]\sim \mathcal{O}(M^{-L}),}
\end{align*}
where {\color{black} the first inequality comes from Jensen's inequality, that is $\mathbb{E}[X]^2\leq \mathbb{E}[X^2]$ for arbitrary random variable $X$,} the second inequality is due to the Lipchitz continuity of $f$ and the last one due to \eqref{eqn:eqvar}.
The first term in \eqref{eqn:mse1} is simply the variance from the MC simulation and can be bounded in terms of $N$ as
\begin{align}\label{eqn:VhatP}
\begin{split}
& \mathbb{E}\big[\big(\hat P-\mathbb{E}[\hat P]\big)^2\big] {\color{black} =\mathbb{V}[\hat P] =\frac{1}{N}\mathbb{V}[f(X_{L}(\xi^u))]}\\
&\leq \frac{1}{N}\Big(\mathbb{V}[f({\color{black}X_{L}(\xi^u))}-f(\mathbf{a}^T\mathbf{b})]^{\frac{1}{2}}+\mathbb{V}[f(\mathbf{a}^T\mathbf{b})]^{\frac{1}{2}}\Big)^2\\
&\leq \frac{ 1}{N}\Big(\frac{C_f}{M^{\frac{L}{2}}}(n\nu+\mu)^{\frac{1}{2}}+\mathbb{V}_{\mathbf{a},\mathbf{b}}[f(\mathbf{a}^T\mathbf{b})]^{\frac{1}{2}}\Big)^2\sim \mathcal{O}(N^{-1}).
\end{split}
\end{align}
Overall, as in \cite{Giles08}, the MSE varies in terms of $\frac{1}{M^L}$ and $\frac{1}{N}$. This is still true even if we sample based on the optimal sampling probability \eqref{eqn:optpro}. Meanwhile, the complexity is in terms of $NM^L$, for an integer $N$ to be determined.

Alternatively, it may be possible to obtain the same accuracy at a reduced computational cost, by considering a multilevel MC simulation \cite{Giles08}. For $l\in[L]\bigcup \{0\}$ define as $\hat P_{l}$ the approximation to $f(\mathbf{a}^T\mathbf{b})$ from $M^l$ sampled indices. Further define $\hat Y_l$ as an estimator of $\mathbb{E}[\hat P_l-\hat P_{l-1}]$ using $N_l$ realizations with $l>0$ and similarly $\hat Y_0$ the estimator of $\mathbb{E}[\hat P_0]$ using $N_0$ samples, that is
\begin{align}\label{eqn:hatYl}
\hat Y_l:=\frac{1}{N_l}\sum_{k=1}^{N_l}(\hat P_l^{(k)}-\hat P_{l-1}^{(k)}).
\end{align}
A key point to note is that both $\hat P_l^{(k)}$ and $\hat P_{l-1}^{(k)}$ emerge from the \emph{same} realization, as we discuss in more detail when we describe our Algorithm \ref{alg:Yl}. By the linear property of the expectation it follows immediately that
\begin{align}\label{eqn:hatY}
\mathbb{E}[\hat P_{L}]=\mathbb{E}[\hat P_{0}]+\sum_{l=1}^L \mathbb{E}[\hat P_{l}-\hat P_{l-1}]\approx \hat Y_0+\sum_{l=1}^L \hat Y_l:=\hat Y,
\end{align}
where clearly $\mathbb{E}[\hat P_L]=\mathbb{E}[\hat Y]$. To investigate the performance of the proposed MLMC estimator $\hat Y$ in \eqref{eqn:hatY} we compare the complexity of two estimators $\hat Y$ and $\hat P$ at the same accuracy level.

\begin{theorem}\label{thm:comp} Let $\mathbf{a}$ and $\mathbf{b}$ be two random vectors with length $n$ drawn from different unknown distributions, that is $\mathbf{a}\sim \mathcal{L}_{\mathbf{a}}$ and  $\mathbf{b} \sim \mathcal{L}_{\mathbf{b}}$, and let $f:\mathbb{R}\to \mathbb{R}$ be a Lipschitz function with Lipschitz number $C_f$. Denote by $P$ the term of interest as in \eqref{eqn:P}, and define $\hat P_l$ the corresponding approximation to $f(\mathbf{\mathbf{a}^T\mathbf{b}})$ based on the sketched version of matrix multiplication via $M^l$ many index samples like in \eqref{eqn:innersketch}.

\begin{enumerate} 
\item If there exist independent estimators $\hat Y_l$ as in \eqref{eqn:hatYl} based on $N_l$ Monte Carlo samples, and positive constants $c_1$, $c_2$, $c_3$ such that
\begin{enumerate}
    \item $\mathbb{E}[\hat P_l -P]\leq c_1M^{-\frac{l}{2}}$,
    \item $\mathbb{V}[\hat Y_l]\leq c_2N_l^{-1}M^{-l}$,
    \item the complexity of $\hat Y_l$, denoted by $C_l$, is bounded by $C_l\leq c_3 N_l M^l$,
\end{enumerate}
then there exists a positive constant $c_4$ such that for $\epsilon< e^{-1}$, there are values $L$ and $N_l$ for which the multilevel estimator $\hat Y=\sum_{l=0}^L\hat Y_l$
has an MSE $\mathbb{E}[(\hat Y-P)^2]$ with bound $\epsilon^2$, and computational complexity $$C(\hat Y):=\sum_{l=0}^L C_l\leq c_4\epsilon^{-2}(\log \epsilon)^2.$$

\item Furthermore, define the estimator based on the finest level $L$ and $N$ realisations as in \eqref{eqn:P} with either the optimal sampling probability \eqref{eqn:optpro} (if tractable) or the uniform probability \eqref{eqn:unique}, and suppose 
\begin{enumerate}
    \item the variance for $\hat P$ is bounded by the same constant $c_2$, i.e., $\mathbb{V}[\hat P]\leq c_2N^{-1}$,
    \item the complexity for $\hat P$ is bounded by the same constant $c_3$, i.e., $C(\hat P)\leq c_3NM^L$,
\end{enumerate}
then at the same accuracy $\epsilon^2$, $C(\hat P)\leq c_6 \epsilon^{-4}$, which is much larger than the upper bound of $C(\hat Y)$ when $\epsilon$ is sufficiently small.
\end{enumerate}
\end{theorem}
\begin{proof} 
\begin{enumerate}
\item 
The proof is based on \cite{Giles08}. Accordingly, the MSE for $\hat Y$ is 
\begin{align*}
&\mathbb{E}\big[(\mathbb{E}[P]-\hat Y)^2\big]=(\mathbb{E}[P]-\mathbb{E}[\hat Y])^2+\mathbb{E}\big[\big(\hat Y-\mathbb{E}\big[\hat Y]\big)^2\big]\\
&=(\mathbb{E}[P]-\mathbb{E}[\hat P_L])^2+\mathbb{V}[\hat Y],
\end{align*}
where $L$ is to be determined. If choosing the ceiling
\begin{align}\label{eqn:L}
    L=\big\lceil \frac{\log(2c_1^2\epsilon^{-2})}{\log M} \big\rceil,
\end{align}
then its bias component can be bounded via condition 1.(a)-(b) as
\begin{align*}
 \bigl (\mathbb{E}[P]-\mathbb{E}[\hat P_L] \bigr )^2\leq c_1^2 M^{-L}\leq \frac{1}{2}\epsilon^2.
\end{align*}

On the other hand, choosing 
\begin{align}\label{eqn:Nl}
  N_l=\lceil 2(L+1)c_2\epsilon^{-2} M^{-l}\big\rceil  
\end{align}
together with condition 1.(b) gives that
\begin{align*}
\mathbb{V}[\hat Y] & \leq \sum_{l=0}^L\mathbb{V}[\hat Y_l]\leq c_2\sum_{l=0}^L N_l^{-1}M^{-l}\\
&\quad \leq c_2\sum_{l=0}^L \big(2(L+1)c_2\epsilon^{-2} M^{-l}\big)^{-1}M^{-l}\\
&\quad = c_2\sum_{l=0}^L \frac{\epsilon^2}{2(L+1)c_2}=\frac{1}{2}\epsilon^2.
\end{align*}
To bound the complexity $C$, let us first find the bound for $L$ in terms of $\log \epsilon^{-1}$. Indeed, $L+1$, defined in \eqref{eqn:L} is bounded by
\begin{align}\label{eqn:L+1}
    L+1\leq  \frac{2\log(\epsilon^{-1})}{\log M}+ \frac{\log(2c_1^2)}{\log M}+2\leq c_5 \log \epsilon^{-1},
\end{align}
where $c_5=\frac{1+\big(0\vee\log(2c_1^2)\big) }{\log M}+2$ given that $\log \epsilon^{-1}>1$ ($\epsilon\leq e^{-1}$). 
Besides, from \eqref{eqn:L} we can get an upper bound for $M^{L-1}$ as 
\begin{align}\label{eqn:ML-1}
M^{L-1}\leq M^{\frac{\log(2c_1^2\epsilon^{-2})}{\log M}}=e^{\log M \frac{\log(2c_1^2\epsilon^{-2})}{\log M}} =2c_1^2\epsilon^{-2}.
\end{align}
Therefore the computational complexity $C$ is bounded through
\begin{align*}
    &C\leq c_3\sum_{l=0}^LN_lM^{l}\leq c_3\sum_{l=0}^L\big( 2(L+1)c_2\epsilon^{-2} M^{-l}+1\big)M^{l}\\
    &\quad = c_3 \Big(2(L+1)^2c_2\epsilon^{-2}+\frac{M^2M^{L-1}-1}{M-1}\Big)\leq c_4\epsilon^{-2}(\log \epsilon)^2,
\end{align*}
where $c_4=2c_2c_3c_5^2+\frac{2c_3c_1^2M^2}{M-1}$.

\item For both estimators $\hat Y$ and $\hat P$, the bias is fixed for the same choice of $L$ in \eqref{eqn:L}. Now let us choose an appropriate $N$ such that $\mathbb{V}[\hat{P}]\leq \frac{1}{2}\epsilon^2$. Let $N=\lceil 2c_2\epsilon^{-2}\rceil$ to meet the accuracy specification, and recall the upper bound for $M^{L-1}$ in \eqref{eqn:ML-1}. Then the complexity $C(\hat P)$ is 
\begin{align*}
    C(\hat P)\leq c_3NM^L\leq c_3(2c_2\epsilon^{-2}+1)M^22c_1^2\epsilon^{-2}\leq c_6 \epsilon^{-4},
\end{align*}
where $c_6=2c_1^2c_3M^2(2c_2+e^{-2})$.
\end{enumerate}
\end{proof}
The application of Theorem \ref{thm:comp} relies on its conditions being verified. This is explored in the form of the following corollary. 
\begin{corollary}\label{cor:comp}
Assume the setting in Theorem \ref{thm:comp} {\color{black} and choose a uniform sampling distribution $\xi^u$ as in \eqref{eqn:unique}}. Then we have
\begin{enumerate}
    \item $c_1=C_f^2(n\nu+\mu)$,
    \item $c_2=2C_f^2(M+1)(n\nu+\mu)+2\mathbb{V}_{\mathbf{a},\mathbf{b}}[P]$,
    \item $c_3=1+M^{-1}$.
\end{enumerate}
\end{corollary}
\begin{proof}

\begin{enumerate}
    \item For any $l\in \mathbb{N}\bigcup \{0\}$ we have that
        \begin{align*}
    &\big(\mathbb{E}[f(\mathbf{a}^T\mathbf{b})]-{\color{black}\mathbb{E}[f(X_{l}(\xi^u))]}\big)^2\leq \mathbb{E}[\big(f(\mathbf{a}^T\mathbf{b})-{\color{black}f(X_{l}(\xi^u))}\big)^2\big]\\
&\leq C_f^2 \mathbb{E}[|\mathbf{a}^T(I-S_l S_l^T) \mathbf{b}|^2]\leq C_f^2 M^{-l}(n\nu+\mu),
\end{align*}
where the last inequality holds because of \eqref{eqn:eqvar}.
    \item For any $l>0$ we have that
    \begin{align*}
&\mathbb{V}[\hat P_l-\hat P_{l-1}]\leq \big(\mathbb{V}[\hat P_l-P]^{\frac{1}{2}}+\mathbb{V}[P-\hat P_{l-1}]^{\frac{1}{2}}\big)^2\\
&\leq \big(\mathbb{E}[(\hat P_l-P)^2]^{\frac{1}{2}}+\mathbb{E}[(\hat P_{l-1}-P)^2]^{\frac{1}{2}}\big)^2\\
&\leq C_f^2\big(\mathbb{E}[|\mathbf{a}^T(I-S_lS^T_l )\mathbf{b}|^2]^{\frac{1}{2}}+\mathbb{E}[|\mathbf{a}^T(I-S_{l-1}S^T_{l-1}) \mathbf{b}|^2]^{\frac{1}{2}}\big)^2\\
&\leq  2C_f^2(M^{-l}+M^{-l+1})(n\nu+\mu) \leq 2C_f^2(M+1)(n\nu+\mu)M^{-l}. 
\end{align*}
For $l=0$ we have that
\begin{align*}
&\mathbb{V}[\hat P_0]={\color{black} \mathbb{V}[f(X_0(\xi))]\leq \big(\mathbb{V}[f(X_0(\xi))-P]^{\frac{1}{2}}}+\mathbb{V}[P]^{\frac{1}{2}}\big)^2\\
&\leq \big(C_f\mathbb{E}[|\mathbf{a}^T(I-S_0S^T_0 )\mathbf{b}|^2]^{\frac{1}{2}}+\mathbb{V}_{\mathbf{a},\mathbf{b}}[P]^{\frac{1}{2}}\big)^2\\
&\leq \big(C_f(n\nu+\mu)^{\frac{1}{2}}+\mathbb{V}_{\mathbf{a},\mathbf{b}}[P]^{\frac{1}{2}}\big)^2\\
&\leq 2C_f^2(n\nu+\mu)+2\mathbb{V}_{\mathbf{a},\mathbf{b}}[P].
\end{align*}
Besides, from \eqref{eqn:VhatP} we can see that $\mathbb{V}[\hat P]$ is bounded by the same $c_2$.
    \item For any $l>0$ we can see easily the complexity is roughly
    \begin{align*}
        C_l\leq N^{l}(M^l+M^{l-1})=(1+M^{-1})N^{1}M^{l},
    \end{align*}
    while 
       \begin{align*}
        C_0\leq N^{0}M^0\leq (1+M^{-1})N^{0}M^{0}.
    \end{align*}
    Besides, we have for the complexity of $\hat P$ that
           \begin{align*}
        C(\hat P)\leq NM^L\leq (1+M^{-1})NM^{L}.
    \end{align*}
Thus $c_3$ can be set as $1+M^{-1}$.
\end{enumerate}
\end{proof}
\begin{remark}\label{rmk:Lbound}
Asymptotically as $l\to \infty$, we have that $\mathbb{E}[P-\hat P_l]\approx c_1 M^{-\frac{l}{2}}$, and hence
$$\mathbb{E}[\hat P_l-\hat P_{l-1}]\approx (\sqrt{M}-1)c_1 M^{-\frac{l}{2}}\approx (\sqrt{M}-1)\mathbb{E}[P-\hat P_{l}].$$
Similarly to the analysis in Section 4.2 of \cite{Giles08}, this information can be used as an approximate bound: $L$ can be set as the smallest integer such that
\begin{equation}\label{eqn:L_num_bound}
|\hat Y_L |<\frac{1}{\sqrt{2}}(\sqrt{M}-1)\epsilon.
\end{equation} 
By doing this, we might achieve a bias bounded by $\frac{\epsilon^2}{2}$ without evaluating $c_1$.
\end{remark}
\begin{remark}[Optimal $N_l$] \label{rmk:Nl}
To achieve a fixed variance, i.e., $\mathbb{V}[\hat Y]<\frac{1}{2}\epsilon$, the optimal $N_l$ can be chosen as
\begin{equation}\label{eqn:optimalNl}
N_l\approx \Big \lceil 2\epsilon^{-2}\sqrt{V_lM^{-l}}\big(\sum_{j=0}^L\sqrt{V_l M^{l} }\big)\Big \rceil,
\end{equation}
where $V_l$ denotes the variance of a single sample $\hat P_l-\hat P_{l-1}$. 
This result is simply an application of Section 1.3 of \cite{Giles15} or Eqn. (12) in \cite{Giles08} to the `stepsize' $M^{-l}$. The estimation for $N_l$ in \eqref{eqn:optimalNl} is conservative and may induce oversampling. In practice, some scaling factor might be introduced to avoid oversampling (see Section \ref{sec:eg1}).
\end{remark}
\subsection{Optimal $M$}\label{sec:optimalM}
This part explores the methods in \cite{Giles08} in order to find an optimal $M$ that reduces the computational complexity of the estimator even further. With $c_2$ given by Corollary \ref{cor:comp}, $L$ and $N_l$ given in the proof of Theorem \ref{thm:comp}, we can express the complexity of $\hat Y$ in terms of $M$ as
\begin{align*}
   & C(\hat Y)\leq \sum_{l=0}^L C_l\approx  \sum_{l=0}^L N_l(M^l+M^{l-1})\stackrel{\eqref{eqn:Nl}}{\approx} \sum_{l=0}^L c_2 (L+1)(M^l+M^{l-1})\epsilon^{-2} M^{-l}\\
    &\stackrel{c_2}{\approx} \sum_{l=0}^L  (L+1)(M+1)^2M^{-1}\epsilon^{-2} =(L+1)^2(M+1)^2 M^{-1} \epsilon^{-2} \\
    &\stackrel{\eqref{eqn:L}}{\approx} M^{-1}(M+1)^2 \log(M)^{-2}\log(\epsilon)^{2}\epsilon^{-2}= g(M)\log(\epsilon)^{2}\epsilon^{-2},
\end{align*}
where 
\begin{equation}\label{fM}
g(M):=M^{-1}(M+1)^2 \log(M)^{-2}.
\end{equation}

As illustrated in figure \ref{fig1} where we plot $g(M)$ against $M$, $g(M)$ drops sharply for $M<6$ and then starts growing slightly again after $M$ going beyond $12$. The minimum (optimum) is attained at $M=11$, however from our experience using either $M=10$ or $M=12$ does not make a significant difference. We remark that our definition of $g(M)$ in \eqref{fM} differs somewhat from that used in \cite{Giles08}, i.e. in the term $(M+1)^2$, but this does not affect the general trend of $g(M)$ as described above. In the numerical experiments of Section \ref{sec:eg1}, a choice of $M=10$ is used as it was deemed appealing in terms of both the performance and time cost.
\begin{figure}
\begin{center}
  \includegraphics[width=0.75\textwidth]{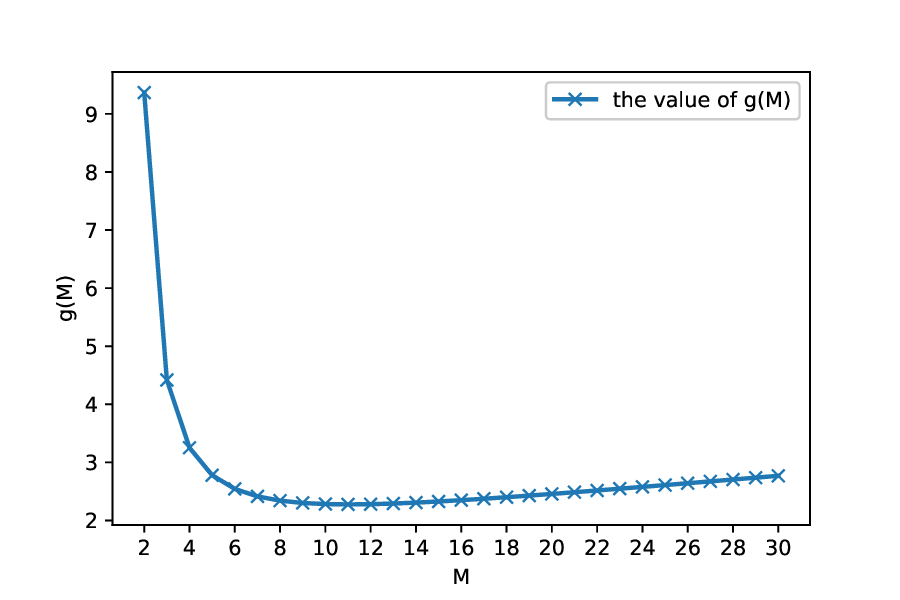} \\
 \caption{The plot of the dominant complexity term $g(M)$ against the base number $M$, indicating the existence of an optimal $M$ at the minimum point.\label{fig1}}
\end{center}
 \end{figure}

\section{Matrix multiplication} \label{sec:matrix}
We now extend our approach to matrix multiplication and thus we consider $A(t)$ and $B(t)$ to be two random matrices of size $m\times n$ and $n\times d$ respectively, $m,n,d\in \mathbb{N}$, drawn from different distributions, elementwise, in the sense $A(t)\sim \mathcal{L}_{A(t)}$ and $B(t) \sim \mathcal{L}_{B(t)}$, and again we suppress $t$ in the notation as in Section \ref{sec:inner} and assume that $n$ is extremely large such that computing directly $AB$ is prohibitively expensive. Recall that $f$ is a Lipchitz function with Lipschitz constant $C_f$, and define $f^{\odot}(AB)$ the elementwise operator on $AB$, that is, 
$$\big(f^{\odot}(AB)\big)_{ik}=f\big((AB)_{ik}\big),\ \ \text{for\ } i\in [m], \text{\ and\ } k\in [d].$$
Once again consider that there is a need to compute $\mathbb{E}_{A,B}[f^{\odot}(AB)]$ where $\mathbb{E}_{A,B}$ is the expectation under $\mathcal{L}_A$ and $\mathcal{L}_B$. 

As in the inner product case, in order to simulate $\mathbb{E}_{A,B}[f^{\odot}(AB)]$ we first approximate $AB$ by random sampling (sketching) for matrix multiplication and then approximate the expectation through a Monte Carlo method. Recall that $\xi:=\{\xi_j\}_{j=1}^{n}$ with $\sum_{j=1}^n \xi_j=1$ is a sampling probability such that an index $j\in [n]$ can be drawn with positive probability $\xi_j$ and $S_L$ a sampling-and-rescaling matrix of size $n\times M^L$ such that 
\begin{align}\label{eqn:matrixsketch}
\widehat{AB}=AS_L S^T_L B =\frac{1}{M^L}\sum_{i=1}^{M^L} \frac{1}{\xi_{r_i}}A_{:,r_i}B_{r_i,:}{\color{black}:=Z_{L}(\xi)}, 
\end{align}
where $A_{:,j}$ indicates the $j$th column of $A$ and $B_{j,:}$ indicates the $j$th row of $B${\color{black}, and $Z_{L}(\xi)$ denotes the matrix-valued random variable that appproximates $AB$ based on $M^L$ indices sampled from $\xi$. 
It is easy to verify that $Z_{L}(\xi)$ is an unbiased estimator under the sampling distribution $\xi$.} 
Besides, following arguments similar to those of the proof of Theorem 2.1 in \cite{Eriksson11} and Lemma 4 in \cite{Drineas06}, we can conclude that the minimum of the expected squared Frobenius error can be achieved by the following result.
\begin{theorem}\label{thm:minvarhigh}
If the matrix elements $A_{ij}$ and $B_{jk}$ are independent random variables, $i\in [m]$, $j\in [n]$ and $k\in [d]$, with finite and nonzero moments $\mathbb{E}_{A}[\|A_{:,j}\|_2^2]$ and $\mathbb{E}_{B}[\|B_{j,:}\|_2^2]$. Then the probability $\xi^{**}$, which is defined as
\begin{align}\label{eqn:optproM}
    \xi^{**}_j=\frac{\sqrt{\mathbb{E}_{A}[\|A_{:,j}\|_2^2]\mathbb{E}_{B}[\|B_{j,:}\|_2^2]}}{\sum_{i=1}^n\sqrt{\mathbb{E}_{A}[\|A_{:,i}\|_2^2]\mathbb{E}_{B}[\|B_{i,:}\|_2^2]}},
\end{align}
minimizes the expected value of the variance in \eqref{eqn:innersketch}, that is,
\begin{align}\label{eqn:minvarhigh}
\begin{split}
  &{\color{black} \min_{\xi}   \mathbb{E}_{A,B}\big[\textbf{E}[\|AB-Z_L(\xi)\|_F^2]\big]= \mathbb{E}_{A,B}\big[\textbf{E}[\|AB-Z_L(\xi^{**})\|_F^2]\big]}\\
   &=\frac{1}{M^L}\Big(\Big(\sum_{j=1}^n\sqrt{\mathbb{E}_{A}[\|A_{:,j}\|_2^2]\mathbb{E}_{B}[\|B_{j,:}\|_2^2]}\Big)^2-\mathbb{E}_{A.B}[\|AB\|_F^2]\Big):=\frac{\bar \mu}{M^L},
   \end{split}
\end{align}
where $\bar \mu=\Big(\sum_{j=1}^n\sqrt{\mathbb{E}_{A}[\|A_{:,j}\|_2^2]\mathbb{E}_{B}[\|B_{j,:}\|_2^2]}\Big)^2-\mathbb{E}_{A,B}[\|AB\|_F^2]$, {\color{black}$\textbf{E}[\cdot]$ denotes the expectation under the distribution $\xi$,} and $\mathbb{E}_{A,B}[\cdot]$ is the expectation with respect to the (element-wise) probabilities of $A$ and $B$.
\end{theorem}
The proof is omitted here as it is quite similar to the proof of Theorem 2.1 in \cite{Eriksson11}. Besides, as discussed in Section \ref{sec:inner}, it is impractical to use $\xi^*$ for random sampling. A simpler option would be to use a uniform probability $\xi^{u}$ as defined in \eqref{eqn:unique}.

\begin{theorem}\label{thm:eqvarhigh}
Assume the same setting as in Theorem \ref{thm:minvarhigh} but with probability $\xi^u$ as defined in \eqref{eqn:unique}, then the expected squared Frobenius error is 
\begin{align}\label{eqn:eqvarhigh}
 {\color{black}   \mathbb{E}_{A,B}\big[\textbf{E}[\|Z_L(\xi^u)-AB\|_F^2]\big]=  \mathbb{E}_{A,B}\big[\textbf{E}[\|Z_L(\xi^{**})-AB\|_F^2]\big]}+\frac{n\bar \nu}{M^l}=\frac{n\bar \nu+\bar \mu}{M^l},
\end{align}

where 
$$\bar \nu=\sum_{i=1}^n\Big(\sqrt{\mathbb{E}_{A}[\|A_{:,i}\|_2^2]\mathbb{E}_{B}[\|B_{i,:}\|_2^2]}-\frac{1}{n}\sum_{j=1}^n\sqrt{\mathbb{E}_{A}[\|A_{:,j}\|_2^2]\mathbb{E}_{B}[\|B_{j,:}\|_2^2]}\Big)^2.$$
\end{theorem}
The proof is omitted here as it is very similar to that of Theorem 2.3.

In this context, a quantity of interest $P$ can be approximated with standard MC coupled to a random sampling method for matrix multiplication via either uniform probability \eqref{eqn:unique} or the optimal probability \eqref{eqn:optproM} (if tractable)
\begin{align}\label{eqn:Phigh}
\mathbb{E}_{A,B}[P]:=\mathbb{E}_{A,B}[f^{\odot}(AB)]\approx\frac{1}{N}\sum_{j=1}^N f(A S_L^{(j)} (S^{(j)}_L)^TB){\color{black}=\frac{1}{N}\sum_{j=1}^N f\big(Z^{(j)}_L(\xi)\big)}:=\hat P,
\end{align}
where $N$ is the number of realisations for $M^L$ many index samples. To consider the MSE for the estimate $\hat P$, we apply a matrix \emph{vectorization}: for instance, if $A\in \mathbb{R}^{m\times n}$,
\begin{align}\label{eqn:vec}
\text{vec}(A)=\text{vec}([A_{:,1}\quad \cdots\quad A_{:,n}])=\begin{bmatrix} A_{:,1}\\
\vdots\\
A_{:,n}
\end{bmatrix}\in \mathbb{R}^{m n},
\end{align}
is the column concatenation of $A$ into a vector. Then the MSE would be
\begin{align}\label{eqn:msehigh}
\begin{split}
\mathbb{E}\big[\|\text{vec}(\hat P-\mathbb{E}[P])\|^2_2\big] & =\mathbb{E}\big[\|\text{vec}(\mathbb{E}[\hat P]-\mathbb{E}[P])\|_2^2\big] +\mathbb{E}\big[\|\text{vec}(\hat P-\mathbb{E}[\hat P])\|^2_2\big]\\
&=\|\text{vec}\big(\mathbb{E}[A(I-S_LS_L^T)B]\big)\|^2_2 +\mathbb{E}\big[\|\text{vec}\big(\hat P-\mathbb{E}[\hat P]\big)\|^2_2\big]\\
&=\|\text{vec}\big(\mathbb{E}[A(I-S_LS_L^T)B]\big)\|^2_2 +\mathbb{V}_{\Vert}\big[\text{vec}(\hat P)\big],
\end{split}
\end{align}
where $\mathbb{E}$ is short for $\mathbb{E}_{A,B,\xi}$ and $\mathbb{V}_{\Vert}\big[\text{vec}(X)\big]:=\mathbb{E}\big[\|\text{vec}\big(X-\mathbb{E}[X]\big)\|^2_2\big]$ for any random matrix $X$. Besides, it is easy to verify that
\begin{align}\label{eqn:vectorization_variance}
 \mathbb{V}_{\Vert}[X+Y]^{\frac{1}{2}}\leq \mathbb{V}_{\Vert}[X]^{\frac{1}{2}} +\mathbb{V}_{\Vert}[Y]^{\frac{1}{2}},
\end{align}
for any random vectors $X$ and $Y$. Note that the variance of a vectorized random matrix is indeed the variance of the random matrix in Frobenius norm. For example,
\begin{align*}
  \mathbb{V}_{\Vert}\big[\text{vec}(\hat P)\big] & =\mathbb{E}\big[\|\text{vec}\big(\hat P-\mathbb{E}[\hat P]\big)\|_2^2\big]=\mathbb{E}\big[\sum_{h=1}^{md}\text{vec}\big(\hat P-\mathbb{E}[\hat P]\big)_h^2\big]\\
  &=\mathbb{E}\big[\sum_{i=1}^{m}\sum_{k=1}^d\big(\hat P-\mathbb{E}[\hat P]\big)_{ik}^2\big]=\mathbb{E}[\big\|\hat P-\mathbb{E}[\hat P]\big\|_F^2].
\end{align*}

With these preliminaries let us now extend the approach of Section \ref{sec:numericalinner} to matrix multiplication. For $l\in [L]\bigcup \{0\}$, define $\hat P_{l}$ as the approximation to $f^{\odot}(AB)$ with $M^l$ many index samples. Recall that $\hat Y_l$ is an estimator of $\mathbb{E}[\hat P_l-\hat P_{l-1}]$ using $N_l$ realizations with $l>0$ and $\hat Y_0$ the respective estimator of $\mathbb{E}[\hat P_0]$ using $N_0$ samples, as defined in \eqref{eqn:hatYl}. Eqn. \eqref{eqn:hatY} remains unchanged, from where we have that $\mathbb{E}[\hat P_L]=\mathbb{E}[\hat Y]$. 

\begin{theorem}\label{thm:comphigh} Let $A$ and $B$ be two random matrices with sizes $m\times n$ and $n\times d$ respectively, drawn from different distributions, namely $A\sim \mathcal{L}_{A}$ and $B \sim \mathcal{L}_{B}$. Let $f:\mathbb{R}\to \mathbb{R}$ be a Lipchitz function with Lipchitz number $C_f$. Denote by $P$ the term of interest as in \eqref{eqn:Phigh}. Define $\hat P_\ell$ the corresponding approximation to $f^{\odot}(AB)$ based on the sketched version of matrix multiplication via $M^\ell$ many index samples like in \eqref{eqn:matrixsketch}.
\begin{enumerate}
\item If there exist independent estimators $\hat Y_l$ as in \eqref{eqn:hatYl} based on $N_l$ Monte Carlo samples, and positive constants $c_1$, $c_2$, $c_3$ such that
\begin{enumerate}
    \item $\big\|\mathrm{vec}(\mathbb{E}[\hat P_l -P])\big\|_2^2\leq c_1^2M^{-l}$,
    \item $\mathbb{V}_{\Vert}[\mathrm{vec}(\hat Y_l)]\leq c_2N_l^{-1}M^{-l}$,
    \item the complexity of $\hat Y_l$, denoted by $C_l$, is bounded by $C_l\leq c_3 N_l M^l$,
\end{enumerate}
then there exists a positive constant $c_4$ such that for $\epsilon< e^{-1}$, there are values $L$ and $N_l$ for which the multilevel estimator $\hat Y=\sum_{l=0}^L\hat Y_l$
has an MSE $\mathbb{E}[\|\mathrm{vec}(\hat Y-\mathbb{E}[P])\|_2^2]$ with bound $\epsilon^2$, with computational complexity $$C(\hat Y):=\sum_{l=0}^L C_l\leq c_4\epsilon^{-2}(\log \epsilon)^2.$$

\item Furthermore, define the estimator based on the finest level $L$ and $N$ realizations as in \eqref{eqn:P} with either the uniform probability \eqref{eqn:unique} or the optimal probability \eqref{eqn:optproM} (if approachable). Suppose 
\begin{enumerate}
    \item the variance for $\hat P$ is bounded by the same constant $c_2$, i.e., $\mathbb{V}_{\Vert}[\hat P]\leq c_2N^{-1}$,
    \item the complexity for $\hat P$ is bounded by the same constant $c_3$, i.e., $C(\hat P)\leq c_3NM^L$,
\end{enumerate}
then with the same accuracy $\epsilon^2$, $C(\hat P)\leq c_6 \epsilon^{-4}$ which is much larger than the bound of $C(\hat Y)$.
\end{enumerate}
\end{theorem}
The proof is similar to that of Theorem \ref{thm:comp}, expect from the decomposition of MSE,
\begin{align}\label{eqn:mseY}
\begin{split}
&\mathbb{E}\big[\|\mathrm{vec}(\hat Y-\mathbb{E}[P])\|^2_2\big]=\mathbb{E}\big[\|\mathrm{vec}(\mathbb{E}[\hat Y-P])\|_2^2\big] +\mathbb{V}_{\Vert}\big[\mathrm{vec}(\hat Y)\big],
\end{split}
\end{align}
so we omit the proof. A more important issue is to verify our proposed MLMC estimator satisfies the conditions of Theorem \ref{thm:comphigh}.
\begin{corollary}\label{cor:comphigh}
Assume the same setting in Theorem \ref{thm:comphigh}  {\color{black} via the sampling distribution $\xi^u$ in \eqref{eqn:unique}}. Then we have
\begin{enumerate}
    \item $c_1=C_f^2(n\bar \nu+\mu)$,
    \item $c_2=2C_f^2(M+1)(n\bar \nu+\mu)+2\mathbb{V}_{\Vert}[f^{\odot}(AB)]$,
    \item $c_3=md(1+M^{-1})$.
\end{enumerate}
\end{corollary}
\begin{proof}

\begin{enumerate}
    \item For any $l\in \mathbb{N}$ we have that
    \begin{align*}
    \big\|\text{vec}\big(\mathbb{E}[f^{\odot}(AB)-f^{\odot}(Z_l(\xi^u))]\big)\big\|_2^2  &\leq \mathbb{E}\big[\big\|\text{vec}\big(f^{\odot}(AB)-f^{\odot}({\color{black}Z_l(\xi^u)})\big)\big\|_2^2\big]\\
&= \mathbb{E}\big[\big\|f^{\odot}(AB)-f^{\odot}({\color{black}Z_l(\xi^u)})\big\|_F^2\big]\\
&\leq C_f^2\mathbb{E}[\|AB-{\color{black}Z_l(\xi^u})\|_F^2] \\
&\leq C_f^2 M^{-l}(n\bar \nu+\mu),
\end{align*}
where the last inequality comes from Theorem \ref{thm:eqvarhigh}.
    \item For any $l>0$ we have that
    \begin{align*}
&\mathbb{V}_{\Vert}[\text{vec}(\hat P_l-\hat P_{l-1})] \\
&\leq \Big(\mathbb{V}_{\Vert}\big[\text{vec}\big(\hat P_l-f^{\odot}(AB)\big)\big]^{\frac{1}{2}}+\mathbb{V}_{\Vert}\big[\text{vec}\big(\hat P_{l-1}-f^{\odot}(AB)\big)\big]^{\frac{1}{2}}\Big)^2\\
&\leq \Big(\mathbb{E}\big[\big\|\text{vec}\big(f^{\odot}(AB)-f^{\odot}({\color{black}Z_l(\xi^u)})\big)\big\|_2^2\big]^{\frac{1}{2}}\\
&\quad +\mathbb{E}\big[\big\|\text{vec}\big(f^{\odot}(AB)-f^{\odot}({\color{black}Z_{l-1}(\xi^u)})\big)\big\|_2^2\big]^{\frac{1}{2}}\Big)^2\\
&\leq 2C_f^2\big(\mathbb{E}[\|AB-{\color{black}Z_{l}(\xi^u)}\|_F^2]^{\frac{1}{2}}+\mathbb{E}[\|AB-{\color{black}Z_{l-1}(\xi^u)}\|_F^2]^{\frac{1}{2}}\big)^2\\
&\leq  2C_f^2(M^{-l}+M^{-l+1})(n\bar \nu+\mu)\\
&\leq 2C_f^2(M+1) (n\bar \nu+\mu)M^{-l}. 
\end{align*}
For $l=0$ we have that
\begin{align*}
&\mathbb{V}_{\Vert}[\text{vec}(\hat P_0)] =\mathbb{V}_{\Vert}\big[\text{vec}\big(f^{\odot}({\color{black}Z_0(\xi^u)})\big)\big]\\
&\leq \Big(\mathbb{V}_{\Vert}\big[\text{vec}\big(f^{\odot}({\color{black}Z_0(\xi^u)})-f^{\odot}(AB)\big)\big]^{\frac{1}{2}}+\mathbb{V}_{\Vert}\big[\text{vec}\big(f^{\odot}(AB)\big)\big]^{\frac{1}{2}}\Big)^2\\
&\leq 2C_f^2\mathbb{E}[\|A(I-S_0S^T_0 )B\|_F^2]+2\mathbb{V}_{\Vert}\big[\text{vec}\big(f^{\odot}(AB)\big)\big]\\
&\leq 2C_f^2(n\bar \nu+\mu)+2\mathbb{V}_{\Vert}\big[\text{vec}\big(f^{\odot}(AB)\big)\big].
\end{align*}
Besides, it is easy to see that $\mathbb{V}_{\Vert}[\text{vec}(\hat P)]$
can be bounded by the same $c_2$ together with $N^{-1}$.
    \item For any $l>0$ we can see easily the complexity is roughly
    \begin{align*}
        C_l\leq md N^{l}(M^l+M^{l-1})=md(1+M^{-1})N^{1}M^{l},
    \end{align*}
    while 
           \begin{align*}
        C_0\leq md N^{0}M^0\leq md(1+M^{-1})N^{0}M^{0}.
    \end{align*}
    Besides, we have for the complexity of $\hat P$ that
           \begin{align*}
        C(\hat P)\leq mdNM^L\leq md(1+M^{-1})NM^{L}.
    \end{align*}
Thus $c_3$ can be set as $md(1+M^{-1})$.
\end{enumerate}
\end{proof}

\begin{remark}\label{rmk:LdetermineMatrix}
Asymptotically as $l\to \infty$, we have that $\|\text{vec}(\mathbb{E}[P-\hat P_l])\|_2\approx c_1 M^{-\frac{l}{2}}$, and hence
$$\|\text{vec}(\mathbb{E}[\hat P_l-\hat P_{l-1}])\|_2\approx (\sqrt{M}+1)c_1 M^{-\frac{l}{2}}\approx (\sqrt{M}+1)\|\text{vec}(\mathbb{E}[P-\hat P_{l}])\|_2.$$
Similarly as in Section 4.2 of \cite{Giles08}, this information can be used as an approximate bound: $L$ can be set as the smallest integer such that
\begin{equation}\label{eqn:L_num_bound_matrix}
\|\text{vec}(\hat Y_L)\|_2<\frac{1}{\sqrt{2}}(\sqrt{M}+1)\epsilon.
\end{equation} 
By doing this, we might achieve a bias bounded by $\frac{\epsilon^2}{2}$ without evaluating $c_1$.
\end{remark}
\begin{remark}[Optimal $N_l$]\label{rmk:NlMatrix}
To achieve a fixed variance, i.e., $\mathbb{V}_{\Vert}[\text{vec}(\hat Y)]<\frac{1}{2}\epsilon$, the optimal $N_l$ can be chosen as
\begin{equation}\label{eqn:optimalNl_matrix}
N_l\approx\Big \lceil 2\epsilon^{-2}\sqrt{V_lM^{-l}}\big(\sum_{j=0}^L\sqrt{V_l M^{l} }\big)\Big \rceil,
\end{equation}
where $V_l$ is the variance of the vectorized form of a single sample $\hat P_l-\hat P_{l-1}$ (recall the definition of the vectorized matrix variance right before \eqref{eqn:vectorization_variance}). 
This result is simply an application of Section 1.3 of \cite{Giles15} or Eqn. (12) in \cite{Giles08} to the `stepsize' $M^{-l}$. 
\end{remark}

To choose the optimal value of $M$ we argue as in Section \ref{sec:optimalM}, that is, $M=11$ leads to the least computational complexity among between all choices of $M$. In the numerical experiments of Section \ref{sec:eg2}, it turns out that $M=10$ and $M=6$ both yield acceptable approximations.


\subsection{MLMC sketching algorithm}\label{sec:numericalinner} 
Based on the general discussion in the beginning of Section 2, Remark \ref{rmk:LdetermineMatrix} and Remark \ref{rmk:NlMatrix}, we propose an algorithm for estimating the matrix product based on MLMC method in Algorithm \ref{alg:MLMCmatrix}. The inner product case discussed in Section 2 can be treated as a special case. Algorithm \ref{alg:MLMCmatrix} approximates $\mathbb{E}[f^{\odot}(AB)]$ through \eqref{eqn:hatY} under uniform probability \eqref{eqn:unique}, where the evaluation for each $\hat Y_l$ in \eqref{eqn:hatY} is performed through function \emph{level\_estimation} described in Algorithm \ref{alg:Yl}. To ensure the convergence, the choices of $L$ and $N_l$ with $l\in [L]\bigcup\{0\}$ are determined within Algorithm \ref{alg:MLMCmatrix} using a \emph{while} loop with one of the conditions given by Eqn. \eqref{eqn:L_num_bound_matrix} in Remark \ref{rmk:LdetermineMatrix}\footnote{The condition for inner product is slightly different, see  Eqn. \eqref{eqn:L_num_bound} in Remark \ref{rmk:Lbound}.}. It is worth noticing that, while the value $L$ and therefore $N_l$ for $l\in [L]\bigcup\{0\}$ are updated in the while loop (see Line 16 and Line 9), previous evaluations for $\hat Y_l$ are reused in Line 13 for efficiency.

Although the outline in Algorithm \ref{alg:Yl} is simple to follow we draw the reader's attention to Line 17 describing how $\hat P_l^{(k)}$ and $\hat P_{l-1}^{(k)}$ are computed through the common realization of $M^l$ indices. Indeed, the procedure for getting $\hat P_{l}^{(k)}$ is by random sampling as in \eqref{eqn:matrixsketch} via the indices of a sample realization of size $M^l$ under uniform probability, and likewise $\hat P_{l-1}^{(k)}$ via $M^{l-1}$ of those $M^l$ indices. That is, taking {\color{black}$(r_1,\ldots,r_{M^l})$} as a realization, then
\begin{align*}
    \hat P_{l}^{(k)}=f^{\odot}\Big(\frac{n}{M^l}\sum_{j=1}^{M^l}A^{(k)}_{:,r_j}B^{(k)}_{r_j,:}\Big),\ \ \text{and}\ \ \  \hat P_{l-1}^{(k)}=f^{\odot}\Big(\frac{n}{ M^{l-1}}\sum_{j=1}^{M^{l-1}}A^{(k)}_{:,r_{jM}}B^{(k)}_{r_{jM},:}\Big).
\end{align*}
{\color{black}Note that in practice the above computation can be further simplified by mapping $(r_1,\ldots,r_{M^l})$ into a set with non-repeated elements.}

\begin{algorithm}
\caption{function \emph{level\_estimation}.}
\label{alg:Yl}
\begin{algorithmic}[1]
\State \textbf{Pre-defined:} $\mathbf{\mathcal{L}_A}$ and $\mathbf{\mathcal{L}_B}$, the distributions of the targeted random matrices.
\State \hspace{1.1cm} $f$, the targeted function; {\color{black} $\xi^u$, the uniform sampling distribution defined in \eqref{eqn:unique}}.
\State \textbf{input:} $l$, the level size;
\State \hspace{1.1cm} $M$, the base number;
\State \hspace{1.1cm} $N_l$, the number of iterations.
\State \textbf{output:} $\hat Y_l$, the approximated version of $\mathbb{E}[\hat P_l-\hat P_{l-1}]$ for $l\neq 0$ or $\mathbb{E}[\hat P_l]$ for $l=0$.
\State \textbf{initialization}: $\hat Y_l=0$.
\If {$l=0$}
 \For {$\ell=1\cdots N_l$}
\State get a pair of samples $A^{(\ell)}$ and $B^{(\ell)}$ from $\mathbf{\mathcal{L}_A}$ and $\mathbf{\mathcal{L}_B}$;
\State sample one index $r$ from $1$ to $n$ {\color{black} according to $\xi^u$};
\State set $\hat Y_l=\hat Y_l+ \frac{1}{N_0}f^{\odot}\big(nA^{(\ell)}_{:,r}B^{(\ell)}_{r,:}\big)$;
\EndFor 
\Else 
\For {$k=1\cdots N_l$}
\State get a pair of samples $A^{(k)}$ and $B^{(k)}$ from $\mathbf{\mathcal{L}_A}$ and $\mathbf{\mathcal{L}_B}$;
\State sample $M^l$ many indices {\color{black} $(r_j)_{j=1}^{M^l}$ from $1$ to $n$ according to $\xi^u$};
\State set $$\hat Y_l=\hat Y_l+ \frac{1}{N_l}\Big(f^{\odot}\Big(\frac{n}{ M^l}\sum_{j=1}^{M^l}A^{(k)}_{:,r_j}B^{(k)}_{r_j,:}\Big)-f^{\odot}\Big(\frac{n}{ M^{l-1}}\sum_{j=1}^{M^{l-1}}A^{(k)}_{:,r_{jM}}B^{(k)}_{r_{jM},:}\Big)\Big);$$
\EndFor 
\EndIf
\State \textbf{return}: $\hat Y_l$.

\end{algorithmic}
\end{algorithm}

\begin{algorithm}
\caption{The MLMC estimator for $\mathbb{E}[f^{\odot}(AB)]$.}
\label{alg:MLMCmatrix}
\begin{algorithmic}[1]
\State \textbf{Pre-defined:} $\mathbf{\mathcal{L}_A}$ and $\mathbf{\mathcal{L}_B}$, the distributions of the targeted random vectors.
\State \hspace{1.1cm} {\color{black}$f$: the targeted function; $\xi^u$: the uniform distribution in \eqref{eqn:unique}}.
\State \textbf{input:} $M$, the base number;
\State \hspace{1.1cm} $\epsilon$, the error tolerance.
\State \textbf{output:} $\hat Y$, the approximated version of $\mathbb{E}[f^{\odot}(AB)]$.
\State \textbf{initialization}: set $L=0$, $t=0$.
\While {$L<3$ or $\frac{\|\hat Y_{L-1}\|_F}{N_{L-1}^{(t-1)}}\geq\frac{1}{\sqrt{2}}(\sqrt{M}+1)\epsilon$}
\State initialize $\hat Y_L=0$;
\State update $V_l$ (defined in Remark \ref{rmk:NlMatrix}) for all $l\in [L]\bigcup \{0\}$; 
\State calculate the optimal $N_l^{(t)}$ for all $l\in [L]\bigcup \{0\}$ through Eqn.\eqref{eqn:optimalNl_matrix};
\State update $\hat Y_L=\hat Y_L+N_L^{(t)}\text{\emph{level\_estimation}}(L,M,N_L^{(t)})$;
\If {$L>0$}
\For {$l=0\cdots L-1$}
\State update $\hat Y_l=\hat Y_l+(N_L^{(t)}-N_L^{(t-1)})\text{\emph{level\_estimation}}(L,M,N_L^{(t)}-N_L^{(t-1)})$;
\EndFor
\EndIf
\State set $L=L+1$ and $t=t+1$;
\EndWhile
\State update $\hat Y_l=\hat Y_l/N_l^{(t-1)}$ for all $l\in [L-1]\bigcup \{0\}$;
\State \textbf{return}: $\sum_{l=0}^{L-1}\hat Y_l$.
\end{algorithmic}
\end{algorithm}

\section{Numerical experiments}\label{sec:num}
In this part, we present some numerical experiments designed to test the performance of the Algorithm \ref{alg:MLMCmatrix} in comparison with a standard MC method embedded with the optimal sampling distribution (see Theorem \ref{thm:minvar} and Theorem \ref{thm:minvarhigh}).  {\color{black}Our experiments are implemented in Python (version 3.6.9) with Numpy-based calculations being optimized under openBLAS \cite{Zhang12} and executed on a Linux cluster with two 14-core E5-2690 v4 Intel Xeon CPUs at 2.60GHz and non-uniform memory allocation.}

\subsection{Example for the inner product}\label{sec:eg1}
Set $n=10^4$ with $\mathbf{a}_j\sim \frac{j}{50}(0.4-N(0,1))$ and $\mathbf{b}_j\sim\cos\big(\text{Poi}(10)+2\text{Exp}(1)\big)\text{Bern}(0.05)$, $j\in [n]$, where $\text{Poi}(\lambda)$ is a Poisson random variable with parameter $\lambda$, $\text{Exp}(\alpha)$ is an exponential random variable with parameter $\alpha$, and $\text{Bern}(\beta)$ is a Bernoulli random variable with success rate $\beta$. As the Bernoulli random variable has low success rate we expect $\mathbf{b}$ to be a sparse vector. In this example we targeted function is set to $f(x):=\frac{1}{|x|H(x+0.4)+0.01}$, where $H(\cdot)$ is an Heaviside step function. It is easy to see in this case $f$ is highly nonlinear. 

We test the Algorithm \ref{alg:MLMCmatrix} for the inner product case with a parameter $M=10$ and error tolerance $\epsilon=0.1${\color{black}, where the reference solution is obtained through direct MC computation:}
\begin{equation}\label{eqn:reference1}
{\color{black}\mathbb{E}[\mathbf{a}^T\mathbf{b}]\approx\frac{1}{\mathcal{N}_1}\sum_{j=1}^{\mathcal{N}_1}(\mathbf{a}^{(j)})^T\mathbf{b}^{(j)},\ \text{with\ } \mathcal{N}_1=10^5.}
\end{equation} 
The value of $L$ and the number of realizations at each level $l\leq L$, i.e., $N_l$ with $l\in [L]\bigcup\{0\}$ are tuned automatically by the algorithm itself. In our case $L=3$ and $N_l$ is obtained through scaling \eqref{eqn:optimalNl} by $\frac{1}{20}$. This scaling factor $\frac{1}{20}$ is introduced to prevent oversampling. Note that the scaling factor does not affect the trend of $N_l$. Figure \ref{fig2} illustrates the trend of the variance of each single path sample $\hat P_l-\hat P_{l-1}$ together with its corresponding $N_l$. From there it is easy to see that there is a clear decay in variance with respect to $l$ from $l=1$, which results in the nearly polynomial decay in the number of realizations $N_l$.
\begin{figure}
\begin{center}
  \includegraphics[width=0.75\textwidth]{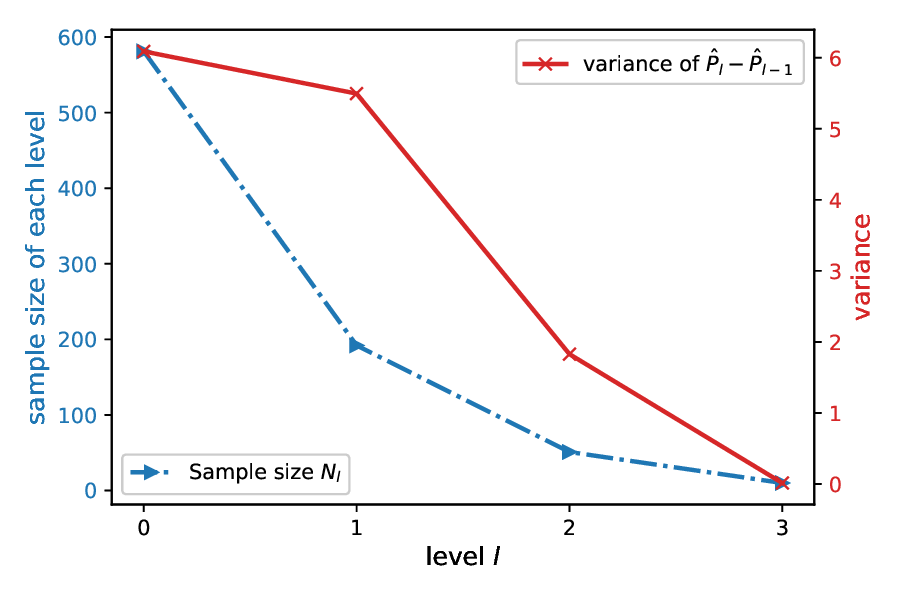} \\
 \caption{Inner product case: plot of the variance of a single realization $\hat P_l-\hat P_{l-1}$ up to $l=L$ (red solid line), and its corresponding number of realizations for each $l$ up to $l=L$ (blue dashed line): for $l=0$, the variance of a single realization $\hat P_l-\hat P_{l-1}$ is indeed the variance of a single realization $\hat P_0$.\label{fig2}}
\end{center}
 \end{figure}
For comparison we also perform a standard MC simulation of the same $M$ and $L$ under \emph{optimal sampling} (Theorem \ref{thm:minvar}) with a number of repetitions chosen to maintain roughly the same accuracy level (convergence). 
The results obtained are tabulated in Table \ref{table1}.
\begin{table}
\centering

  \begin{tabular}{c||c|c|c|c|c|c|c}
\hline
\hline
 &
\multicolumn{3}{c|}{\color{black}MLMC using $\xi^{u}$} &
\multicolumn{3}{c|}{\color{black} MC using $\xi^{*}$} &\color{black} directMC \\
\hline
($M$,$L$) & AE & RE  & time cost & AE & RE  & time cost & \color{black} time cost\\
\hline
$(10,3)$ & 0.002 & 0.079 & 0.047 s & 0.001 & 0.041 & 0.274 s &\color{black} 0.423 s\\
\hline
\hline
\end{tabular}
\caption{Numerical results from the implementation of our method on approximating the inner product. These include records of the relative errors (RE), absolute errors (AE) and computational times for Algorithm \ref{alg:MLMCmatrix} under $M=10$ and its corresponding $L$. {\color{black}For comparison we provide also the results from standard MC \eqref{eqn:P} with optimal sampling distribution $\xi^{*}$  \eqref{eqn:optpro} based on the finest level $L$, and time cost for getting reference solution through direct MC \eqref{eqn:reference1}.}}
\label{table1}
\end{table}

From Table \ref{table1}, the MLMC estimator using $\xi^{u}$ in general outperforms the MC one using $\xi^{*}$ in terms of the elapsed time. Though MC using $\xi^{u}$ provides an approximation that doubles the accuracy of MLMC, its computational time is about six times longer. {\color{black}The computation times of both estimators are less than that of the directMC for getting the reference solution, which illustrates the advantage of our proposed estimator in practice.}  

\subsection{Example for the matrix multiplication}\label{sec:eg2}
In the matrix multiplication case we consider a setup with $n=10^4$, $m=d=10^3$ using $A_{ij}\sim g_1\big(\frac{j}{10^4}(0.5-N(0,1))\big), $ where $g_1(x):=\sin(x)+N(0,1)x$, and $B_{jk}\sim g_2(\text{Poi}(2))\text{Bern}(0.2)$, where $g_2(x):=\cos(x)H(5-x)$ for $i\in[m], j\in [n]$ and $k\in[d]$. Like before, $\text{Poi}(\lambda)$ denotes a Poisson random variable with parameter $\lambda$ and $\text{Bern}(\beta)$ is a Bernoulli random variable with success rate $\beta$. The targeted function is chosen to be $f(x):=|x|H(2-x)$, where $H(\cdot)$ is an Heaviside step function.

\begin{figure}
\begin{center}
  \includegraphics[width=0.75\textwidth]{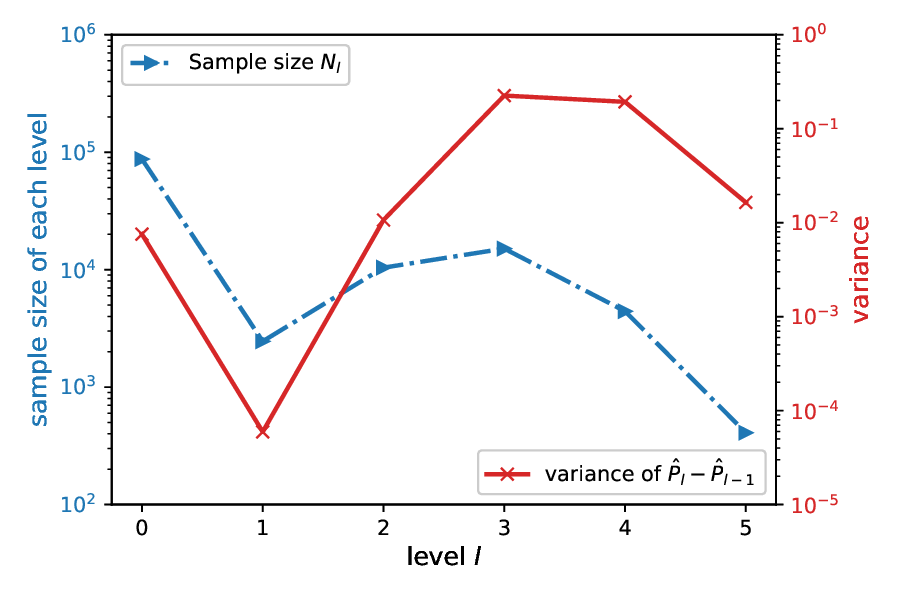} \\
 \caption{Matrix multiplication case: Plot of the variance of a single realization $\hat P_l-\hat P_{l-1}$ up to $l=L$ (red solid line), and its corresponding number of realizations for each $l$ up to $l=L$ (blue dashed line): for $l=0$, the variance of a single realization $\hat P_l-\hat P_{l-1}$ is indeed the variance of a single realization $\hat P_0$.\label{fig3}}
\end{center}
 \end{figure}

Similar to the inner product example, we run Algorithm \ref{alg:MLMCmatrix} for the matrix product with base number $M=10$ and the error tolerance $\epsilon=0.1$. The reference solution is computed as
\begin{equation}\label{eqn:reference2}
{\color{black}\mathbb{E}[A^TB]\approx\frac{1}{\mathcal{N}_2}\sum_{j=1}^{\mathcal{N}_2}A^{(j)} B^{(j)},\ \text{with\ } \mathcal{N}_2=10^5.}
\end{equation} 
Algorithm \ref{alg:MLMCmatrix} automatically chooses $L=5$. Though $M^L$ is now larger than $n$, this does not imply sampling all the columns of $A$. Besides, the number of realizations generated at the finest level $L$ is very small, which does not affect the total performance. Meanwhile the variance of each single path sample $\hat P_l-\hat P_{l-1}$ together with its corresponding $N_l$ is directly obtained through \eqref{eqn:optimalNl_matrix}. Figure \ref{fig3} illustrates the trend of the variance of $\hat P_l-\hat P_{l-1}$ and $N_l$ up to $l=L$. It can be noted from Figure \ref{fig3} that from $l=3$ the variance curve begins to decay, while the apparent low values in variance from $l=0$ to $l=2$ are due to the nature of randomised sketching of matrix multiplication. For example, to get an approximation {\color{black}$Z_l(\xi^u)$} of $AB$ through \eqref{eqn:matrixsketch} with $l=0$, only one column of A and the corresponding row of B are selected and multiplied. This is guaranteed to decrease the variance of $Z_0(\xi^u)$. We can also observe this phenomenon from the inner product case as depicted in Figure \ref{fig2}, where the variance of $l=0$ is only slightly bigger than the one of $l=1$. {\color{black} The curve in Figure \ref{fig3} also indicates that our proposed estimator will be more efficient in super-large-scale matrix application, where the variance decay speeds up for higher level $l>>5$. }

To compare performance, {\color{black}a standard MC simulation of the same $M$ and $L$, formulated in \eqref{eqn:Phigh},} is implemented with optimal sampling distribution $\xi^{**}$ (Theorem \ref{thm:minvarhigh}) with the number of repetitions chosen to maintain roughly the same accuracy level. As matrix $B$ is a very sparse matrix, optimal probability defined in Eqn. \ref{eqn:minvarhigh} might have sparse or very small entries. Therefore even $M^L$ is now larger than $n$, the probability that all the columns of $A$ are sampled to obtain an approximation is pretty small. The results obtained are recorded in Table \ref{table2}. {\color{black} The MLMC estimator using $\xi^{u}$} in general outperforms the {\color{black} MC one using $\xi^{**}$} in terms of the elapsed time. {\color{black}Meanwhile, the computational times for MC using $\xi^{**}$  are admittedly very large, taking three times longer than the directMC. This is mainly due to the choice of the high level $L=5$ compared to the matrix size. On the other hand, it is reasonable to anticipate that the MLMC method under the approximated optimal probability instead of the uniform one, would lead to a drastic improvement of the efficiency of the approximation beyond what has been demonstrated in this work.} 
\begin{table}
\centering

  \begin{tabular}{c||c|c|c|c|c|c|c}
\hline
\hline
 &
\multicolumn{3}{c|}{\color{black} MLMC using $\xi^{u}$} &
\multicolumn{3}{c|}{\color{black} MC using $\xi^{**}$} & \color{black} directMC\\
\hline
($M$,$L$) & AE & RE  & time cost & AE & RE  & time cost & \color{black}time cost \\
\hline

$(10,5)$ & 0.088 & 0.006  & \color{black}2.240s & 0.069 &  0.005 & \color{black} 75.173 s & \color{black} 25.561 s\\
\hline
\hline
\end{tabular}
\caption{Numerical results from the implementation of our method on approximating the matrix product. These include records of the absolute errors in Frobenius norm (AE), the relative errors (RE) and computational times for Algorithm \ref{alg:MLMCmatrix} under $M=10$ and its corresponding $L$. {\color{black} For the sake of comparison we provide also the results from a standard MC simulation \eqref{eqn:Phigh} based on the finest level $L$ and sampling distribution $\xi^{**}$ \eqref{eqn:optproM}, and the timecost for getting the reference solution through direct MC \eqref{eqn:reference2}}.} 
\label{table2}
\end{table}



\section{Conclusions}

We presented a new approach for computing arbitrary vector and matrix products `on-the-fly' that combines ideas from sketching in randomized linear algebra and multilevel Monte Carlo approaches for estimating high-dimensional integrals. Our approach is simple to implement and, subject to optimizing some algorithmic parameters, it outperforms the standard Monte Carlo in both in terms of the accuracy and the time required for computing the estimator.

\section*{Acknowledgements}
The authors are grateful to EPSRC for funding this work through the project EP/R041431/1, titled `Randomness: a resource for real-time analytics'; YW is also funded by The Alan Turing Institute under the EPSRC grant EP/N510129/1 and by EPSRC though the project EP/S026347/1, titled 'Unparameterised multi-modal data, high order signatures, and the mathematics of data science'.

\bibliographystyle{alpha}

\footnotesize
\begin{thebibliography}{99}

\bibitem{Beskos} Beskos, A. , Jasra, A.,Law K., Tempone, R., and Zhou, Y. (2017). {\em Multilevel sequential Monte Carlo samplers}. Stochastic Processes and their Applications, 127(5), 1417-1440. 

\bibitem{Bierig} Bierig, C. and Chernov, A.(2015). {\em Convergence analysis of multilevel Monte Carlo variance estimators and application for random obstacle problems}. Numerische Mathematik, 130(4), 579-613.

\bibitem{Drineas06} Drineas, P., Kannan R, and Mahoney, W. M. (2006). {\em Fast Monte Carlo algorithms for matrices I: Approximating matrix multiplication}, SIAM J. Comput. 36(1), 132-157. 

\bibitem{Eriksson11} Eriksson-Bique, S., Solbrig, M., Stefanelli, M., Warkentin, S., Abbey, R., Ipsen, I.C.F. (2011). {\em Importance sampling for a Monte Carlo matrix multiplication algorithm, with application to information retrieval}, SIAM J. Comput.,33, 1689–1706. 

\bibitem{Giles08}Giles, M. B. (2008). {\em Multilevel monte carlo path simulation.} Operations Research 56.3, 607-617.


\bibitem{Giles09} Giles, M. B., and Waterhouse, B. J. (2009). {\em Multilevel quasi-Monte Carlo path simulation.} Advanced Financial Modelling, Radon Series on Computational and Applied Mathematics 8, 165-181.

\bibitem{Giles15} Giles, M. B. (2015) {\em Multilevel monte carlo methods.} Acta Numerica 24, 259-328.

\bibitem{HI15} Holodnak, J., and Ipsen, I. (2015). {\em Randomized approximation of the gram matrix: Exact computation and probabilistic bounds.} SIAM Journal on Matrix Analysis and Applications 36.1: 110-137.

\bibitem{Kebaier05} Kebaier, A. (2005). {\em Statistical Romberg extrapolation: a new variance reduction method and applications to option pricing}.  The Annals of Applied Probability 15.4 (2005): 2681-2705.

\bibitem{Kar12} Kar, P., and Karnick, H. (2012). {\em Random feature maps for dot product kernels}. In Artificial Intelligence and Statistics, 583-591.

\bibitem{Rahimi08} Rahimi, A., and Recht, B. (2008). {\em Random features for large-scale kernel machines.} In Advances in neural information processing systems, 1177-1184.

\bibitem{Teckentrup} Teckentrup, A.L., Scheichl, R., Giles,  M. B., and Ullmann E. (2013). {\em Further analysis of multilevel Monte Carlo methods for elliptic PDEs with random coefficients}. Numerische Mathematik, 125(3), 569-600.

\bibitem{Wu18} Wu, Y. (2018). {\em A Note on Random Sampling for Matrix Multiplication.} arXiv preprint,  arXiv:1811.11237.

\bibitem{Zhang12} {\color{black}Zhang, X., Wang, Q. and Chothia, Z. (2012). {\em openBLAS.} http://xianyi.github. io/OpenBLAS, 88.}

\end{thebibliography}

\end{document}